\documentclass[a4paper,12pt]{amsart}
\usepackage[colorlinks=true, linkcolor=blue, citecolor=blue]{hyperref}
\usepackage{amscd,amssymb,epsfig,amsxtra,amsfonts}
\usepackage{amsmath}
\usepackage{stmaryrd, mathrsfs}
\usepackage{ upgreek }
\usepackage{verbatim}
\usepackage[utf8]{inputenc}
\usepackage{epsfig} 
\newcommand{\ms}[1]{\mbox{\tiny$#1$}}
\newcommand{\epsh}[2]
          {\begin{array}{c} \hspace{-1.3mm}
         \raisebox{-4pt}{\epsfig{figure=fig/#1,height=#2}}
         \hspace{-1.9mm}\end{array}}

\newcommand{\coev}{\operatorname{coev}}

\newcommand{\tev}{\stackrel{\longleftarrow}{\operatorname{ev}}}

\newcommand{\rcoev}{\stackrel{\longrightarrow}{\operatorname{coev}}}
\newcommand{\rev}{\stackrel{\longrightarrow}{\operatorname{ev}}}
\newcommand{\lev}{\stackrel{\longleftarrow}{\operatorname{ev}}}
\newcommand{\lcoev}{\stackrel{\longleftarrow}{\operatorname{coev}}}

\newcommand{\Id}{\operatorname{Id}}
\newcommand{\qdim}{\operatorname{qdim}}

\newcommand{\End}{\operatorname{End}}
\newcommand{\Hom}{\operatorname{Hom}}

\newcommand{\ptr}{\operatorname{ptr}}

\newcommand{\wb}{\overline}

\newcommand{\cat}{\operatorname{\mathscr{C}}}

\newcommand{\tv}{\operatorname{t}}

\newcounter{exo} \newcounter{numexercice}
\renewcommand{\theexo}{\arabic{exo}}

\newtheorem{theo}{Theorem}[section]
\newtheorem{Def}[theo]{Definition}
\newtheorem{The}[theo]{Theorem}
\newtheorem{Lem}[theo]{Lemma}

\newtheorem{Pro}[theo]{Proposition}
\newtheorem{HQ}[theo]{Corollary}
\newtheorem{rmq}[theo]{Remark}
\begin{document}
\title[Anomaly-free TQFTs from $\mathfrak{sl}(2|1)$]{Anomaly-free TQFTs from the super Lie algebra $\mathfrak{sl}(2|1)$}


\author{Ngoc Phu HA}        
\address{Hung Vuong University, Faculty of Natural Science, Nong Trang, Viet Tri, Phu Tho, Viet Nam}
\email{ngocphu.ha@hvu.edu.vn}


\maketitle

\begin{abstract}
It is known that the category $\mathscr{C}^H$ of nilpotent weight modules over the quantum group associated with the super Lie algebra $\mathfrak{sl}(2|1)$ is a relative pre-modular $G$-category. Its modified trace enables to define an invariant of $3$-manifolds. In this article we show that the category $\mathscr{C}^H$ is a relative modular $G$-category which allows one to construct a family of non-semi-simple extended topological quantum field theories which surprisingly are anomaly free. The quantum group associated with $\mathfrak{sl}(2|1)$ is considered at odd roots of unity.
\end{abstract}
\vspace{25pt}



\subsection*{Acknowledgments}
The result of the article is a part of my PhD thesis. The author would like to thank my supervisor, B. Patureau-Mirand for his useful comments and for his encouragement.

\section{Introduction}
By renormalization the invariant of Reshetikhin-Turaev, the authors N. Geer and B. Patureau-Mirand presented the notion of a modified trace in \cite{NgBpVt09}. This notion gave them a family of link invariants (see \cite{NgBpVt09, NgBPm12}) and  of invariant of $3$-manifolds (see \cite{FcNgBp14}). The invariants of $3$-manifolds in  \cite{FcNgBp14} have been used to define a family of non-semi-simple topological quantum field theories (TQFTs) in \cite{CFNP16}.

A relative pre-modular $G$-category $\mathscr{C}$ is a $G$-graded $\Bbbk$-linear ribbon category having a modified trace and two abelian groups $(G,Z)$ in which $Z$ acts on a class of objects of $\mathscr{C}$ satisfying some compatible conditions, see Definition \ref{G-modrel} (see also \cite{FcNgBp14}).
A category of finite weight modules over a quantization of $\mathfrak{sl}(2)$ have been proven to be relative pre-modular in \cite{FcNgBp14}, and this category is used to illustrate the main results of the article \cite{FcNgBp14}.
 Such categories are the ingredients to construct topological invariants $N_r$ or CGP invariants of the computable triple $(M, T, \omega)$ where $T$ is a $\mathscr{C}$-colored ribbon graph in $M$ and $\omega\in H^1(M\setminus T, G)$ (see \cite[Theorem 4.7]{FcNgBp14}). In the article \cite{Mar17}, the author enriched the structure of a relative pre-modular category $\cat$ by adding a relative modularity condition, the category $\cat$ with this condition is called a relative modular category, see Definition \ref{Def of modular category}. Relative modular categories are main ingredients to construct a family of non-semisimple extended topological quantum field theories (ETQFTs) \cite{Mar17}. A model of a relative modular category from the unrolled quantum group associated to the Lie algebra $\mathfrak{sl}(2)$ was first constructed in \cite{CFNP16}. 

Working with the unrolled quantum group associated with the super Lie algebra $\mathfrak{sl}(2|1)$, denoted by $\mathcal{U}_{\xi}^{H}\mathfrak{sl}(2|1)$ where $\xi$ is a root of unity, one see that the category $\cat^H$ of nilpotent weight modules over $\mathcal{U}_{\xi}^{H}\mathfrak{sl}(2|1)$ is a relative pre-modular $G$-category \cite{Ha16}.
This fact allowed one to construct an invariant of CGP type.
In this paper, we show for the super Lie algebra $\mathfrak{sl}(2|1)$  that, as in the case of Lie algebras, $\cat^H$ is a modular $G$-category relative to $(G,Z)$ where $\mathit{G}=\mathbb{C} / \mathbb{Z} \times \mathbb{C}/ \mathbb{Z}$ and $Z=\mathbb{Z}\times \mathbb{Z}\times \mathbb{Z}/2\mathbb{Z}$. 
This fact provides a new family of examples of relative modular categories. One specific feature with the $\mathfrak{sl}(2|1)$ case is that the anomaly of the TQFT happen to be trivial.

The paper is organized in four sections. In Section 2, we recall some definitions about the relative pre-modular $G$-categories and the quantum group $\mathcal{U}_{\xi}^{H}\mathfrak{sl}(2|1)$. Then in Section 3, we describe the category $\cat^H$ of nilpotent finite dimensional modules over $\mathcal{U}_{\xi}^{H}\mathfrak{sl}(2|1)$ and its properties. Finally, Section 4 proves the category $\cat^H$ is a relative modular $G$-category.

\section{Preliminaries}
In this section we recall some definitions and results about a relative pre-modular $G$-category from \cite{FcNgBp14}, and a quantization of the super Lie algebra $\mathfrak{sl}(2|1)$. 
\subsection{Relative pre-modular $G$-categories}
The categories we mention in the paper are ribbon. 
A {\em tensor category} $\cat$ is a category equipped with a covariant bifunctor $- \otimes -: \cat \times\cat \rightarrow \cat$ called the tensor product, a unit object $\mathbb{I}$, an associativity constraint, and left and right unit constraints such that the Triangle and Pentagon Axioms hold (see \cite[XI.2]{ChKa95}).

A {\em braiding} on a tensor category $\cat$ consists of a family of isomorphisms $\{c_{V,W}: V\otimes W \rightarrow W\otimes V\}$, defined for each pair of objects $V,W$ which satisfy the Hexagon Axiom \cite[XIII.1 (1.3-1.4)]{ChKa95} as well as the naturality condition expressed in the commutative diagram \cite[XIII.1.2]{ChKa95}. We say a tensor category is {\em braided} if it has a braiding. 

A {\em pivotal category} is a tensor category which has duality if for each object $V \in \cat$ there exits an object $V^*$ and
morphisms
\begin{align*}
&\rev_{V}: V^{*} \otimes V \rightarrow \mathbb{I}, &\rcoev_{V}: \mathbb{I} \rightarrow V \otimes V^{*},\\
&\lev_{V}:V \otimes V^{*} \rightarrow \mathbb{I}, &\lcoev_{V}:\mathbb{I} \rightarrow V^{*} \otimes V
\end{align*} 
satisfying the relations
\begin{equation*}
(\Id_V\otimes \rev_V)\circ (\rcoev_V\otimes \Id_V)=\Id_V,\ (\rev_V\otimes \Id_{V^*})\circ (\Id_{V^*} \otimes \rcoev_V)=\Id_{V^*},
\end{equation*}
and
\begin{equation*}
(\lev_V\otimes \Id_{V})\circ (\Id_{V}\otimes \lcoev_V)=\Id_V,\ (\Id_{V^*}\otimes \lev_V)\circ (\lcoev_V\otimes \Id_{V^*})=\Id_{V^*}.
\end{equation*}
A {\em twist} in a braided tensor category $\cat$ with duality is a family $\{\theta_V: V\rightarrow V \}$ of natural
isomorphisms defined for each object $V$ of $\cat$ satisfying relations \cite[XIV.3.1-3.2]{ChKa95}.

A {\em ribbon category} is a braided tensor category with duality and a twist. We say that $\mathscr{C}$ is a $\Bbbk$-linear category if for all $V,W \in \mathscr{C}$, the morphisms $\Hom_{\mathscr{C}}(V,W)$ form a $\Bbbk$-vector space and the composition and the tensor product are bilinear and, $\End_{\mathscr{C}}(\mathbb{I})\cong \Bbbk$ where $\Bbbk$ is a field.

 Let $\mathscr{C}$ be a $\Bbbk$-linear ribbon category. A set of objects of $\mathscr{C}$ is said to be commutative if for any pair $\{V, W \}$ of these objects, we have $c_{V,W}\circ c_{W,V}=\Id_{W \otimes V}$ and $\theta_{V}=\Id_{V}$. Let $(Z, +)$ be a commutative group. A {\em realization} of $Z$ in $\mathscr{C}$ is a commutative set of objects $\{{\varepsilon^{t}}\}_{t \in Z}$ such that $\varepsilon^{0}=\mathbb{I}, \qdim(\varepsilon^{t})=1$ and $\varepsilon^{t} \otimes \varepsilon^{t'}=\varepsilon^{t+t'}$ for all $t, t' \in Z$.

 A realization of $Z$ in $\mathscr{C}$ induces an action of $Z$ on isomorphism classes of objects of $\mathscr{C}$ by $(t, V) \mapsto \varepsilon^{t} \otimes V$. We say that $\{{\varepsilon^{t}}\}_{t \in Z}$ is a {\em free realization} of $Z$ in $\mathscr{C}$ if this action is free. This means that $\forall t \in Z \backslash \{0\}$ and for any simple object $V \in \mathscr{C}, V \otimes \varepsilon^{t} \not\simeq V$. We call {\em simple $Z$-orbit} the reunion of isomorphism classes of an orbit for this action.
  
We call {\em a modified trace} $\tv$ on ideal $\texttt{Proj}$ of projective objects of $\mathscr{C}$ a family of linear maps $\{\tv_V: \End_{\cat}(V)\rightarrow \Bbbk\}_{V\in \texttt{Proj}}$ satisfying the conditions
$\forall U,V \in \texttt{Proj}, \forall W \in \mathscr{C}$,
$$\forall f \in \Hom_{\mathscr{C}}(U,V), \forall g \in \Hom_{\mathscr{C}}(V, U), \tv_{V}(f \circ g)= \tv_{U}(g \circ f)$$
$$\forall f \in \End_{\mathscr{C}}(V \otimes W), \ \tv_{V\otimes W}(f)=\tv_{V}(\ptr_{R}(f))$$
where $\ptr_{R}(f)=(\Id_{V}\otimes \lev_{W})\circ (f \otimes \Id_{W^{*}})\circ (\Id_{V} \otimes \rcoev_{W}) \in \End_{\cat}(V)$. We call $d(V)=\tv_V(\Id_V)$ the {\em modified dimension} of $V\in \cat$. 
The formal linear combination $\Omega_g=\sum_i d(V_i)V_i$ is called {\em a Kirby color of degree $g \in G$} if the isomorphism classes of the $\{V_i\}_i$ are in one to one correspondence with the simple $Z$-orbits of $\cat_g$.

Recall that $F$ is the Reshetikhin-Turaev functor from the ribbon category $\mathcal{R}ib_{\cat}$ of ribbon $\cat$-colored graphs to $\cat$ (see \cite{Tura94}). The renormalization of $F$ is denoted by $F'$ (see \cite{NgBpVt09}).
 \begin{Def}[\cite{FcNgBp14}] \label{G-modrel}
 Let $(\mathit{G},\times)$ and $(Z, +)$ be two commutative groups. A $\Bbbk$-linear ribbon category $\mathscr{C}$ is a pre-modular $\mathit{G}$-category relative to $\mathcal{X}$ with modified dimension {\em d} and periodicity group $Z$ if 
 \begin{enumerate}
 \item  the category $\mathscr{C}$ has a $\mathit{G}$-grading $\{\mathscr{C}_g\}_{g \in \mathit{G}}$,
 \item  the group $Z$ has a free realization $\{{\varepsilon^{t}}\}_{t \in Z}$ in $\mathscr{C}_{1}$ (where $1 \in \mathit{G}$ is the unit),
 \item  there is a $\mathbb{Z}$-bilinear application $\mathit{G}\times Z \rightarrow \Bbbk^{\times}, (g,t)\mapsto g^{\bullet t}$ such that $\forall V \in \mathscr{C}_{g}, \forall t \in Z, c_{V,\varepsilon^{t}}\circ c_{\varepsilon^{t}, V}=g^{\bullet t}\Id_{\varepsilon^{t} \otimes V}$,
 \item there exists $\mathcal{X} \subset \mathit{G}$ such that $\mathcal{X}^{-1}=\mathcal{X}$ and $\mathit{G}$ cannot be covered by a finite number of
translated copies of $\mathcal{X}$, in other words $\forall g_{1}, ..., g_{n} \in \mathit{G}, \cup_{i=1}^{n}(g_{i}\mathcal{X}) \neq \mathit{G}$,
 \item for all $g \in \mathit{G}\setminus \mathcal{X}$, the category $\mathscr{C}_{g}$ is semi-simple and its simple objects are in the reunion of a finite number of simple $Z$-orbits,
 \item  there exists a nonzero trace $t$ on ideal $\texttt{Proj}$ of projective objects of $\mathscr{C}$ and {\em d} is the associated modified dimension,
 \item  there exists an element $g \in \mathit{G}\setminus \mathcal{X}$ and an object $V \in \mathscr{C}_{g}$ such that the scalar $\Delta_{+}$ defined in Figure \ref{Kirby colour} is nonzero; similarly, there exists an element $g \in \mathit{G}\setminus \mathcal{X}$ and an object $V \in \mathscr{C}_{g}$ such that the scalar $\Delta_{-}$ defined in Figure \ref{Kirby colour} is nonzero,
 \begin{figure}
   \centering
   $$
   \begin{array}{ccc}
     F\left( \epsh{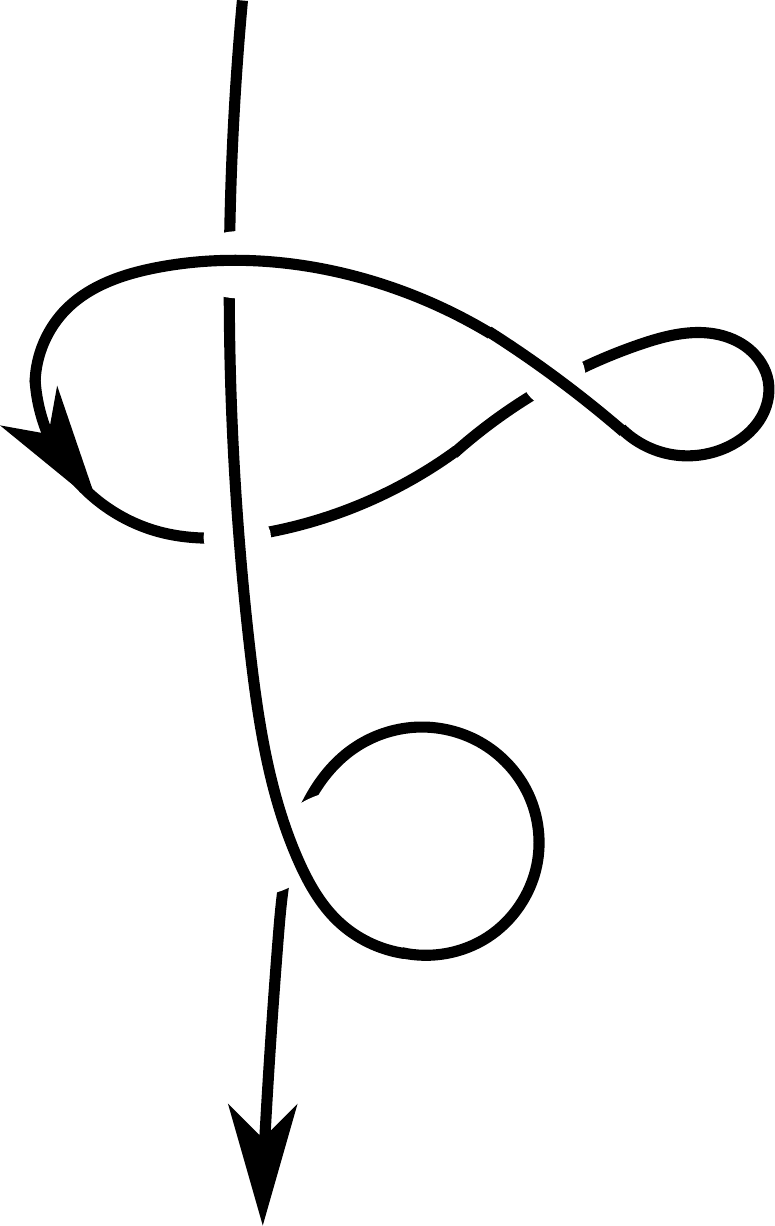}{10ex}
 		\put(-35,2){\ms{\Omega_{\overline{\mu}}}}
 		\put(-20,-22){\ms{V}}\right) = \Delta_{-}\Id_{V}, \ 
 	F\left( \epsh{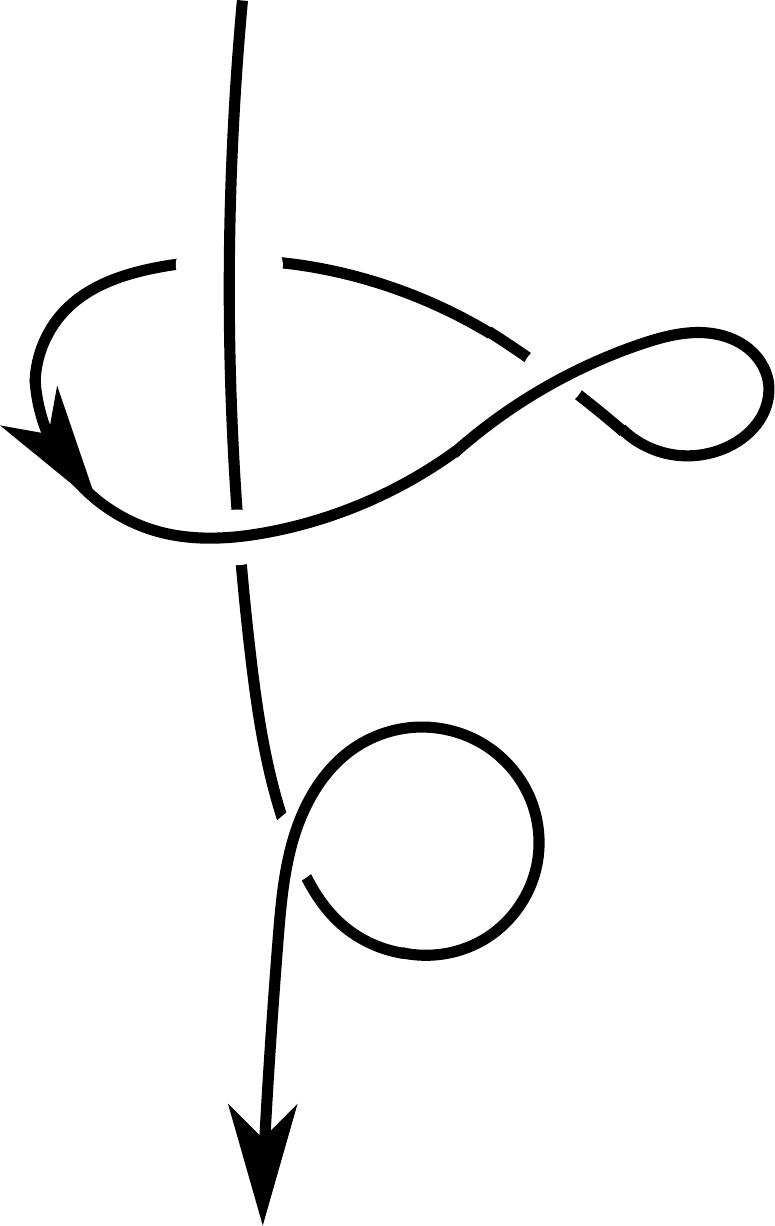}{10ex}
 		\put(-35,2){\ms{\Omega_{\overline{\mu}}}}
 		\put(-20,-22){\ms{V}}\right) = \Delta_{+}\Id_{V}
   \end{array}
   $$
	\caption{{$V \in \mathscr{C}_{g}$ and $\Omega_{\overline{\mu}}$ is a Kirby color of degree $\mu$.}}
	\label{Kirby colour}
 \end{figure} 
  \item the morphism $S(U, V)=F(H(U,V)) \neq 0 \in \End_{\mathscr{C}}(V)$, for all simple objects 	$U,V \in \texttt{Proj}$, where $$H(U,V)= \epsh{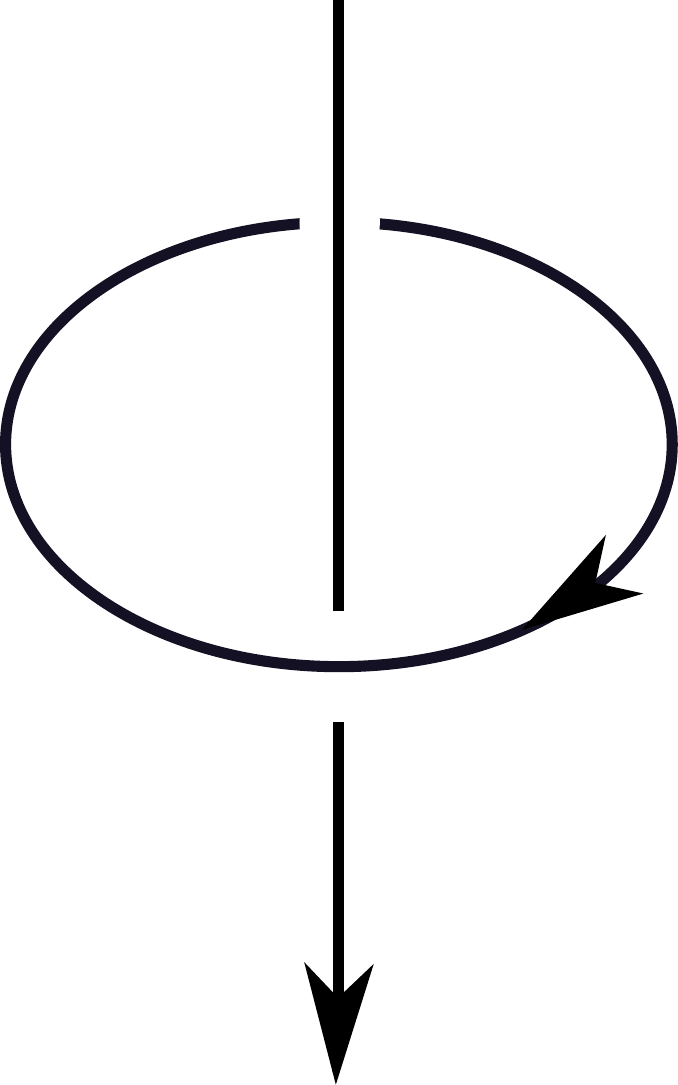}{7ex}
      \put(-6,-7){\ms{U}}
			\put(-10,18){\ms{V}} \in \End_{\mathscr{C}}(V).$$
 \end{enumerate}
  \end{Def}
  An example of a relative pre-modular $G$-category $\cat$ is the category of finite dimensional weight modules over $\wb U_q^{H}\mathfrak{sl}(2)$, see \cite[Subsection 6.3]{FcNgBp14}. Note that in \cite{FcNgBp14} the word {\em relative modular} is used instead of {\em relative pre-modular}. A relative pre-modular category is anomaly free if $\Delta_+ = \Delta_-$. In this case it is always possible to renormalize so that $\Delta_+ = \Delta_- = 1$.  
 Next we consider a quantization of the super Lie algebra $\mathfrak{sl}(2|1)$ from \cite{Ha16}.
\subsection{Quantum group $\mathcal{U}_{\xi}^H\mathfrak{sl}(2|1)$}
\begin{Def}
Let $\ell\geq 3$ be an odd integer and $\xi=\exp(\frac{2\pi i}{\ell})$.
The superalgebra $\mathcal{U}_\xi\mathfrak{sl}(2|1)$ is an associative superalgebra on $\mathbb{C}$ generated by the elements $k_1,k_2,k_1^{-1},k_2^{-1}, e_1,e_2,f_1,f_2$ 
and the relations
\begin{align*}
&k_1k_2=k_2k_1,\\
&k_{i}k_{i}^{-1}=1, \ i=1,2,\\
&k_ie_jk_i^{-1}=\xi^{a_{ij}}e_j, k_if_jk_i^{-1}=\xi^{-a_{ij}}f_j \ i,j=1,2,\\
&e_1f_1-f_1e_1=\frac{k_1-k_1^{-1}}{\xi-\xi^{-1}}, e_2f_2+f_2e_2=\frac{k_2-k_2^{-1}}{\xi-\xi^{-1}},\\
&[e_1, f_2]=0, [e_2, f_1]=0,\\
&e_2^{2}=f_2^{2}=0,\\
&e_1^{2}e_2-(\xi+\xi^{-1})e_1e_2e_1+e_2e_1^{2}=0,\\
&f_1^{2}f_2-(\xi+\xi^{-1})f_1f_2f_1+f_2f_1^{2}=0.
\end{align*}
The last two relations are called the Serre relations. The matrix $(a_{ij})$ is given by $a_{11}=2, a_{12}=a_{21}=-1, a_{22}=0$. The odd generators are $e_2, f_2$. 
\end{Def}
 We define $\xi^{x}:=\exp(\frac{2\pi i x}{\ell})$, and denote
 $$\{x\}= \xi^{x}-\xi^{-x}.$$

The algebra $\mathcal{U}_\xi\mathfrak{sl}(2|1)$ has a structure of a Hopf algebra with the coproduct, counit and the antipode given by (see \cite{SMkVNt91})
\begin{align*}
&\Delta(e_i)=e_i \otimes 1 + k_i^{-1} \otimes e_i \ i=1,2,\\ 
&\Delta(f_i)=f_i \otimes k_i + 1 \otimes f_i \ i=1,2,\\
&\Delta(k_i)=k_i \otimes k_i \ i=1,2,\\
&S(e_i)=- k_ie_i, S(f_i)=-f_ik_i^{-1}, S(k_i)=k_i^{-1} \ i=1,2,\\
&\epsilon(k_i)=1, \epsilon(e_i)=\epsilon(f_i)=0 \ i=1,2.
\end{align*}

We extend $\mathcal{U}_\xi\mathfrak{sl}(2|1)$ to a superalgebra over $\mathbb{C}$, denote by $\mathcal{U}_{\xi}^{H}\mathfrak{sl}(2|1)$, by adding two generators $h_1, h_2$ and the associated relations. This means that $\mathcal{U}_{\xi}^{H}\mathfrak{sl}(2|1)$ is a $\mathbb{C}$-superalgebra generated by $k_i, k_i^{-1}, e_i, f_i$ and $h_i$ for $i=1,2$, and the relations in $\mathcal{U}_\xi\mathfrak{sl}(2|1)$ plus the relations
 $$[h_{i},e_{j}]=a_{ij}e_{j}, [h_{i},f_{j}]=-a_{ij}f_{j}[h_{i},h_{j}]=0, \text{ and } [h_{i},k_{j}]=0$$ for $i, j = 1, 2$.
	
	The superalgebra $\mathcal{U}_{\xi}^{H}\mathfrak{sl}(2|1)$ is a Hopf superalgebra where the coproduct $\Delta$, the antipode $S$ and the counit $\epsilon$ are determined as in $\mathcal{U}_\xi\mathfrak{sl}(2|1)$ and by 
	$$\Delta(h_i)=h_i \otimes 1 + 1 \otimes h_i, S(h_i)=-h_i, \epsilon(h_i)=0 \ i=1,2.$$

Set $e_3=e_1e_2-\xi^{-1}e_2e_1, f_3=f_2f_1-\xi f_1f_2$.
Denote $$\mathfrak{B}_{+}=\{e_{1}^{p^{'}}e_{3}^{\sigma^{'}}e_{2}^{\rho^{'}}, p^{'} \in \{0,1,...,\ell-1\}, \rho^{'},\sigma^{'} \in \{ 0,1 \}\},$$ $$\mathfrak{B}_{-}=\{f_{2}^{\rho}f_{3}^{\sigma}f_{1}^{p}, p \in \{0,1,...,\ell-1\}, \rho, \sigma \in \{ 0,1 \}\},$$  $$\mathfrak{B}_{0}=\{k_{1}^{s_{1}}k_{2}^{s_{2}}, s_{1}, s_{2} \in \mathbb{Z}  \}\ \text{and} \  \mathfrak{B}_{h}=\{h_{1}^{t_{1}}h_{2}^{t_{2}}, t_{1}, t_{2} \in \mathbb{N}  \}.$$
We consider the quotients $\mathcal{U}=\mathcal{U}_\xi\mathfrak{sl}(2|1)/(e_{1}^{\ell}, f_{1}^{\ell})$, this superalgebra has a Poincar\'e-Birkhoff-Witt basis $\mathfrak{B}_{+}\mathfrak{B}_{0}\mathfrak{B}_{-}$ and  $\mathcal{U}^{H}=\mathcal{U}_{\xi}^{H}\mathfrak{sl}(2|1)/(e_{1}^{\ell}, f_{1}^{\ell})$ which has a Poincar\'e-Birkhoff-Witt basis $\mathfrak{B}_{+}\mathfrak{B}_{0}\mathfrak{B}_{h}\mathfrak{B}_{-}$.
\section{Relative pre-modular $G$-category $\mathscr{C}^H$}
In this section we represent first the category $\cat$ of nilpotent finite dimensional representations over $\mathcal{U}$, then the category $\cat^H$ of nilpotent finite dimensional representations over $\mathcal{U}^H$.
\subsection{Category $\cat$ of weight modules over $\mathcal{U}$ }
We consider the even category $\mathscr{C}$ of the nilpotent finite dimensional modules over $\mathcal{U}_{\xi}\mathfrak{sl}(2|1)$, its objects are finite dimensional representations of $\mathcal{U}_{\xi}\mathfrak{sl}(2|1)$ on which $e_{1}^{\ell}=f_{1}^{\ell}=0$ and $k_1, k_2$ are diagonalizable operators. If $V, V^{'} \in \mathscr{C}$, $\Hom_{\mathscr{C}}(V, V^{'})$ is formed by the even morphisms between these two modules (see \cite{NgBp08}). Each nilpotent simple module is determined by the highest weight $\mu = (\mu_{1}, \mu_{2}) \in \mathbb{C}^{2}$ and is denoted $ V_{\mu_{1}, \mu_{2}}$ or $V_{\mu}$. Its highest weight vector $w_{0, 0, 0}$ satisfies
\begin{align*} 
e_1w_{0,0,0}&=0, &e_2w_{0,0,0}&=0, \\
k_1w_{0,0,0}&=\lambda_1 w_{0,0,0}, &k_2w_{0,0,0}&=\lambda_2 w_{0,0,0}
\end{align*}
where $\lambda_i=\xi^{\mu_i}$ for $i=1,2.$
	
	For $\mu=(\mu_{1}, \mu_{2}) \in \mathbb{C}^{2}$ we say that $\mathcal{U}_\xi\mathfrak{sl}(2|1)$-module $V_{\mu}$ is typical if it is a simple module of dimension $4\ell$. Other simple modules are said to be atypical.
	
	The basis of a typical module is formed by vectors $w_{\rho, \sigma, p}=f_2^{\rho}f_3^{\sigma}f_1^{p}w_{0,0,0}$ where $\rho, \sigma \in \{0, 1\}, 0\leq p < \ell$. The odd elements are $w_{0,1, p}$ and $w_{1,0, p}$, others are even. The representation of typical $\mathcal{U}_\xi\mathfrak{sl}(2|1)$-module $V_{\mu_1, \mu_2}$ is determined by
\begin{align*}
&k_{1}w_{\rho, \sigma, p}=\lambda_{1}\xi^{\rho-\sigma-2p}w_{\rho, \sigma, p}, \\
&k_{2}w_{\rho, \sigma, p}=\lambda_{2}\xi^{\sigma+p}w_{\rho, \sigma, p},\\
&f_{1}w_{\rho, \sigma, p}=\xi^{\sigma-p}w_{\rho, \sigma, p+1}-\rho(1-\sigma)\xi^{-\sigma}w_{\rho-1, \sigma +1, p},\\
&f_{2}w_{\rho, \sigma, p}=(1-\rho)w_{\rho+1, \sigma, p}, \\
&e_{1}w_{\rho, \sigma, p}=-\sigma(1-\rho)\lambda_{1}\xi^{-2p+1}w_{\rho+1, \sigma-1, p}+[p][\mu_{1}-p+1]w_{\rho, \sigma, p-1},\\
&e_{2}w_{\rho, \sigma, p}=\rho[\mu_{2}+p+\sigma]w_{\rho -1, \sigma, p}+\sigma(-1)^{\rho}\lambda_{2}^{-1}\xi^{-p}w_{\rho, \sigma -1, p+1}.
\end{align*}
where $\rho, \sigma \in \{0,1 \}$ and $p \in \{0,1,...,\ell-1\}$. \vspace{13pt}

	We also have $V_{\mu}\simeq V_{\mu+\vartheta}\Leftrightarrow \vartheta \in (\ell \mathbb{Z})^{2}$.

\begin{rmq}	 
	The module $V_{\mu}$ is typical if $[\mu_{1}-p+1] \ne 0 \ \forall p \in \{1,...,\ell-1\} \ (\mu_{1} \neq p-1+\frac{\ell}{2}\mathbb{Z}\ \forall p \in \{1,...,\ell-1\})$ and $[\mu_{2}][\mu_{1}+\mu_{2}+1] \ne 0 \ (\mu_2 \neq \frac{\ell}{2}\mathbb{Z}, \mu_{1}+\mu_{2} \neq -1 + \frac{\ell}{2}\mathbb{Z})$ {\em (see \cite{BaDaMb97})}.
\end{rmq}
	
\subsection{Relative pre-modular $G$-category $\mathscr{C}^H$}
Let $V_{\mu_{1},\mu_{2}}$ be an object of $\cat$. We define the actions of $h_i, i=1, 2$ on the basis of $V_{\mu_{1},\mu_{2}}$ by $$h_1w_{\rho, \sigma, p}=(\mu_1+\rho - \sigma - 2p)w_{\rho, \sigma, p}, h_2w_{\rho, \sigma, p}=(\mu_2+\sigma + p)w_{\rho, \sigma, p}.$$
With these actions $V_{\mu_{1},\mu_{2}}$ is a module over $\mathcal{U}^H$.

	We consider the even category $\mathscr{C}^{H}$ of nilpotent finite dimensional $\mathcal{U}^{H}$-modules, that means $e_{1}^{\ell}=f_{1}^{\ell}=0$, and for which $\xi^{h_i}=k_i$ as diagonalizable operators with $i=1,2$. The categories $\cat$ and $\mathscr{C}^{H}$ are pivotal in which the pivotal structure is given by $g = k_1^{-\ell}k_2^{-2}$ (see \cite{Ha16}).

	Thus $V_{\mu_{1},\mu_{2}}$ is a weight module of $\mathscr{C}^{H}.$ A module in $\mathscr{C}^{H}$ is said to be typical if, seen as a $\mathcal{U}_\xi\mathfrak{sl}(2|1)$-module, it is typical. For each module $V$ we denote $\overline{V}$ the same module with the opposite parity.
	We set $\mathit{G}=\mathbb{C} / \mathbb{Z} \times \mathbb{C}/ \mathbb{Z}$ and for each $\overline{\mu} \in \mathit{G}$ we define $\mathscr{C}_{\overline{\mu}}^{H}$ as the subcategory of weight modules which have their weights in the coset $\overline{\mu} \ (\text{modulo} \ \mathbb{Z}\times \mathbb{Z})$. So $\{ \mathscr{C}_{\overline{\mu}}^{H} \}_{\overline{\mu} \in \mathit{G}}$ is a $\mathit{G}$-graduation (where $\mathit{G}$ is an additive group): let $V \in \mathscr{C}_{\overline{\mu}}^{H}, V^{'}\in \mathscr{C}_{\overline{\mu}^{'}}^{H}$, then the weights of $V\otimes V^{'}$ are congruent to $\overline{\mu} + \overline{\mu}^{'}\ (\text{modulo} \ \mathbb{Z}\times \mathbb{Z})$. Furthermore, if $\overline{\mu} \ne \overline{\mu}^{'}$ then $\Hom_{\mathscr{C}^{H}}(V,V^{'})=0$ because a morphism preserves weights. 

It is shown that the category $\cat^H$ is braided, pivotal and have a twist (see \cite{Ha16}). Thus we have the proposition.
\begin{Pro}
$\mathscr{C}^{H}$ is a ribbon category.
\end{Pro}	
The pivotal structure is given by $g=k_{1}^{-\ell}k_{2}^{-2}$ and the braiding is determined by  
\begin{equation*}
\mathcal{R}=\check{\mathcal{R}}\mathcal{K}
\end{equation*} where 
\begin{equation}\label{R-matrix}
\check{\mathcal{R}}=\sum_{i=0}^{\ell-1}\frac{\{1\}^{i}e_1^{i} \otimes f_1^{i}}{(i)_{\xi}!} (1-e_3\otimes f_3)(1-e_2 \otimes f_2),
\end{equation}
in which $(0)_{\xi}!=1, (i)_{\xi}!=(1)_{\xi}(2)_{\xi}\cdots (i)_{\xi}, (k)_{\xi}=\frac{1-\xi^k}{1-\xi}$ and 
\begin{equation}\label{R-matrix 1}
\mathcal{K}=\xi^{-h_1 \otimes h_2 -h_2 \otimes h_1 - 2h_2 \otimes h_2}.
\end{equation}
	In addition, by \cite[Theorem 3.17]{Ha16} the subcategory $\cat_{\wb\alpha}^H$ is semi-simple for $\wb\alpha\in G\setminus \mathcal{X}$ where 
	$$\mathcal{X}=\left\{\overline{0}, \overline{\frac{1}{2}}\right\}\times \mathbb{C}/ \mathbb{Z} \cup \mathbb{C}/ \mathbb{Z} \times \left \{\overline{0}, \overline{\frac{1}{2}}\right\} \cup \left\{(\overline{\mu}_1, \overline{\mu}_2): \overline{\mu}_1 + \overline{\mu}_2\in\left\{\overline{0}, \overline{\frac{1}{2}}\right\}\right\}.$$
By \cite[Theorem 4.4]{Ha16}, there exists a modified trace $\tv$ on the ideal of projective modules of $\cat^H$. It is checked that these datas satisfy the conditions of Definition \ref{G-modrel} and we have the statement.

\begin{Pro}\label{pre modular}
$\mathscr{C}^H$ is pre-modular $G$-category relative to $(\mathit{G}, Z)$ where $\mathit{G}=\mathbb{C} / \mathbb{Z} \times \mathbb{C}/ \mathbb{Z}$ and $Z=\mathbb{Z}\times \mathbb{Z}\times \mathbb{Z}/2\mathbb{Z}$. Furthermore, it is anomaly free, i.e., $\Delta_+ = \Delta_- = 1$.
\end{Pro}
Note that the third component of the group $Z$ sets up the parity of modules $\varepsilon^n, n\in Z$ in the realization of $Z$.
\section{Relative modular $G$-category $\cat^H$}
In this section we show $\mathscr{C}^H$ is relative modular $G$-category. This category allows one to construct an invariant of $3$-manifolds which is a main ingredient of the constructions of ETQFTs (see \cite{Mar17}).
\begin{Def}[\cite{Mar17}]\label{Def of modular category}
A pre-modular $\mathit{G}$-category $\mathscr{C}$ relative to $\mathcal{X}$ with modified dimension {\em d} and periodicity group $Z$ is said a modular $G$-category relative to $(\mathit{G}, Z)$ if it satisfies the modular condition, i.e., it exists a relative modularity parameter $\zeta \in \mathbb{C}^*$ such that
 \begin{equation*}
	d(V_i)f_{ij}^{\overline{\mu}}=
	\begin{cases}
	\zeta (\overrightarrow{\coev}_{V_i}\circ \tev_{V_i})\ &\text{if} \  i=j\\
	0 \ &\text{if} \ i \neq j
	\end{cases}
\end{equation*} 
for all $\overline{\mu}, \overline{\nu} \in \mathit{G}\setminus \mathcal{X}$ and for all $i, j \in \overline{\nu}$ which $V_i, V_j$ are not in the same $Z$-orbit, where $f_{ij}^{\overline{\mu}}$ is the morphism determined by the $\mathscr{C}$-colored ribbon tangle depicted in Figure \ref{Rel modularity} under Reshetikhin-Turaev functor $F$. 
\begin{figure}
   \centering
   $$
   \begin{array}{ccc}
     \epsh{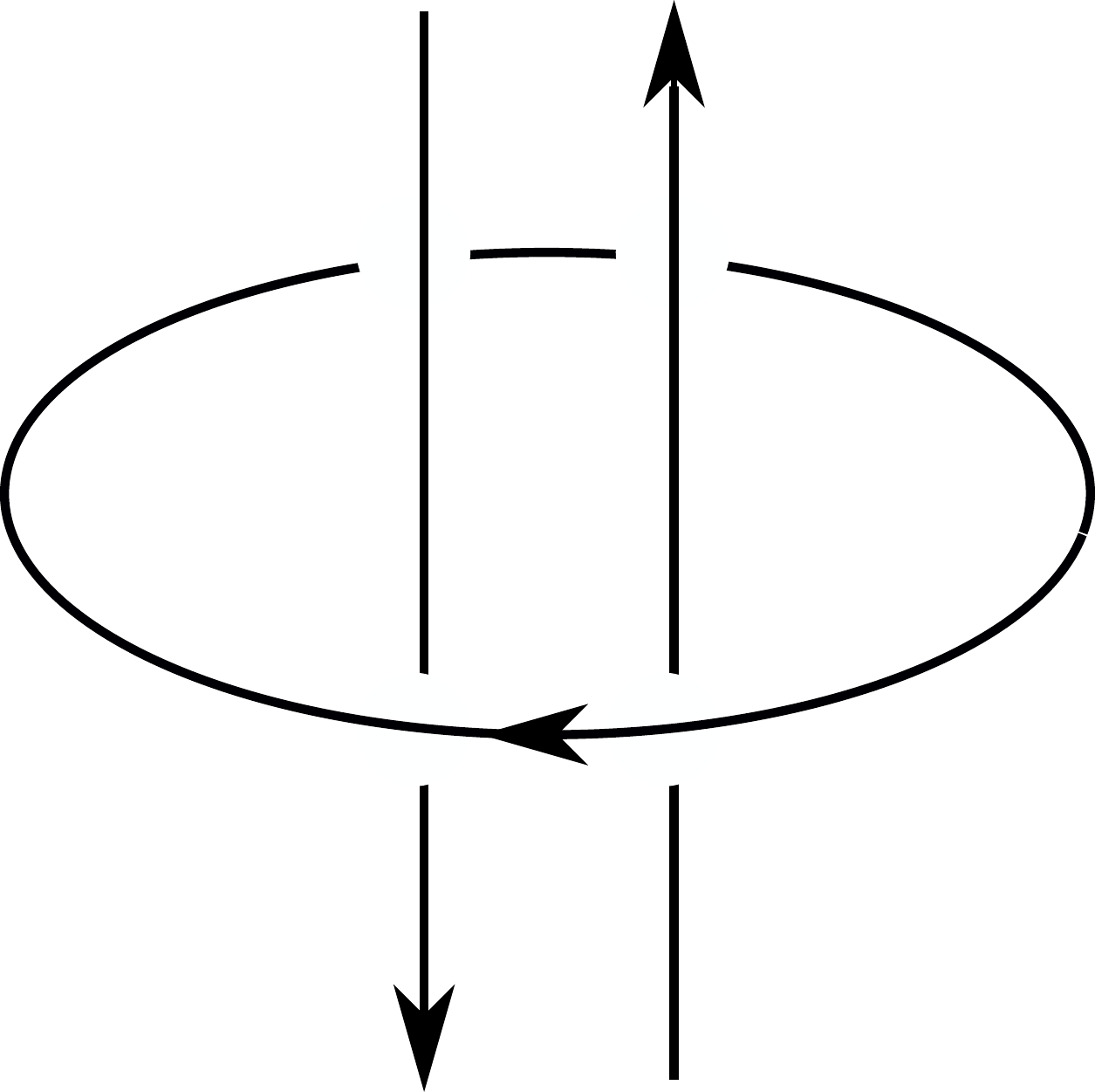}{10ex}
 		\put(3,2){\ms{\Omega_{\overline{\mu}}}}
 		\put(-42,-22){\ms{V_i}}
 		\put(-18,-22){\ms{V_j}}
   \end{array}
   $$
	\caption{Representation of the morphism $f_{ij}^{\overline{\mu}}$}
	\label{Rel modularity}
 \end{figure} 
\end{Def} 
\begin{The}\label{main theorem}
Category $\mathscr{C}^H$ of nilpotent weight modules over $\mathcal{U}^{H}$ is modular $\mathit{G}$-category relative to $(\mathit{G}, Z)$.
\end{The}
To prove the theorem we need the lemmas below.
\begin{Lem}\label{vector UH-invariant}
For $ \overline{\nu} \in \mathit{G}\setminus \mathcal{X}$, let $V_i\in \mathscr{C}_{\overline{\nu}}^H$. 
Then any vector $y$ in the image of $f_{ii}^{\overline{\mu}}$ is $\mathcal{U}^H$-invariant.
\end{Lem}
\begin{proof}
Let $V_k \in \mathscr{C}_{\overline{\nu}}^H$, by \cite[Lemma 4.9]{FcNgBp14} we can do a handle-slide move on the circle component of the graph representing $f_{ii}^{\overline{\mu}} \otimes \Id_{V_k}$ to obtain the equalities
$$c_{W, V_k}\circ (f_{ii}^{\overline{\mu}} \otimes \Id_{V_k})=c_{V_k, W}^{-1}\circ (f_{ii}^{\overline{\mu}+\overline{\nu}} \otimes \Id_{V_k})=c_{V_k, W}^{-1}\circ (f_{ii}^{\overline{\mu}} \otimes \Id_{V_k})$$
where $W=V_i\otimes V_{i}^*$.
The braidings $c_{W, V_k}, c_{V_k, W}^{-1}: W \otimes V_k \rightarrow V_k \otimes W$ are given by
$c_{W, V_k}=\tau^{s}\circ \mathcal{R}$ and $c_{V_k, W}^{-1}=\mathcal{R}^{-1}\circ \tau^{s}$ where $\mathcal{R}=\check{\mathcal{R}}\mathcal{K}$, see Equations \eqref{R-matrix} and \eqref{R-matrix 1}.

Let $x\neq 0$ be a weight vector of weight $0$ of $W$ and $v \in V_k$ be an even weight vector of weight $\nu=(\nu_1, \nu_2)$, set $y=f_{ii}^{\overline{\mu}}(x) \in W$. 

Let $W_{+}^{'}$ be the vector space generated by $\{e_{1}^{i_1}e_{3}^{i_3}e_{2}^{i_2} y \ |\ i_1+i_2+i_3 > 1\ \text{for}\ 0\leq i_1\leq\ell -1,\ 0\leq i_2, i_3 \leq 1\}$, $W_{-}^{'}$ be the vector space generated by $\{f_{1}^{i_1}f_{3}^{i_3}f_{2}^{i_2} v \ |\ i_1+i_2+i_3 > 1\ \text{for}\ 0\leq i_1\leq\ell -1,\ 0\leq i_2, i_3 \leq 1\}$, $V_{+}^{'}$ be the vector space generated by $\{e_{1}^{i_1}e_{3}^{i_3}e_{2}^{i_2} v \ |\ i_1+i_2+i_3 > 1\ \text{for}\ 0\leq i_1\leq\ell -1,\ 0\leq i_2, i_3 \leq 1\}$ and $V_{-}^{'}$ be the vector space generated by $\{f_{1}^{i_1}f_{3}^{i_3}f_{2}^{i_2} y \ |\ i_1+i_2+i_3 > 1\ \text{for}\ 0\leq i_1\leq\ell -1,\ 0\leq i_2, i_3 \leq 1\}$.
Because the weight of $x$ is $0$ then $\mathcal{K}(y \otimes v)=y \otimes v$. Hence 
\begin{align*}
&c_{W, V_k}(y \otimes v)=\tau^{s}\circ \check{\mathcal{R}}\mathcal{K}(y \otimes v)\\
&=v \otimes y +(\xi-\xi^{-1})f_1 v\otimes e_1 y +f_3 v\otimes e_3 y + f_2v\otimes e_2y+W_{-}^{'}\otimes W_{+}^{'}
\end{align*}
and, 
\begin{align*}
&c_{V_k, W}^{-1}(y \otimes v)=\mathcal{R}^{-1}\circ \tau^{s}(y \otimes v)=(S\otimes \Id_{\mathcal{U}^H})(\mathcal{R})(v \otimes y)\\
&=(S\otimes \Id_{\mathcal{U}^H})\left(v\otimes y +(\xi-\xi^{-1})e_1v\otimes f_1y -e_3v \otimes f_3y-e_2v\otimes f_2y + V_{+}^{'}\otimes V_{-}^{'}\right)\\
&=v\otimes y -(\xi-\xi^{-1})k_1e_1v\otimes f_1y +k_1k_2e_3v \otimes f_3y+k_2e_2v\otimes f_2y + S\left(V_{+}^{'}\right)\otimes V_{-}^{'}.
\end{align*}
Identify the two right hands of the equations above, on gets $e_1y=f_1y=0$ and $e_2y=f_2y=0$.
By the relations $e_1f_1-f_1e_1=\frac{k_1-k_1^{-1}}{\xi-\xi^{-1}}, e_2f_2+f_2e_2=\frac{k_2-k_2^{-1}}{\xi-\xi^{-1}}$, it implies that $k_{i}^2y=y$ for $i=1, 2$ and also since $k_i$ act as $\xi^{h_i}$ and the weights of $W$ are in $\mathbb{Z}\times \mathbb{Z}\times \mathbb{Z}/2\mathbb{Z}$, we have that the eigenvalues of $k_i$ are in $\xi^{\mathbb{Z}}$ which does not contain $-1$ (note that $\ell$ is odd). Thus $k_iy=y$ for $i=1, 2$ and $y$ is an invariant vector of $W$.
\end{proof}
Recall that $S'(V,W)$ is the number complex determined by 
$$S'(V,W) =\left< \epsh{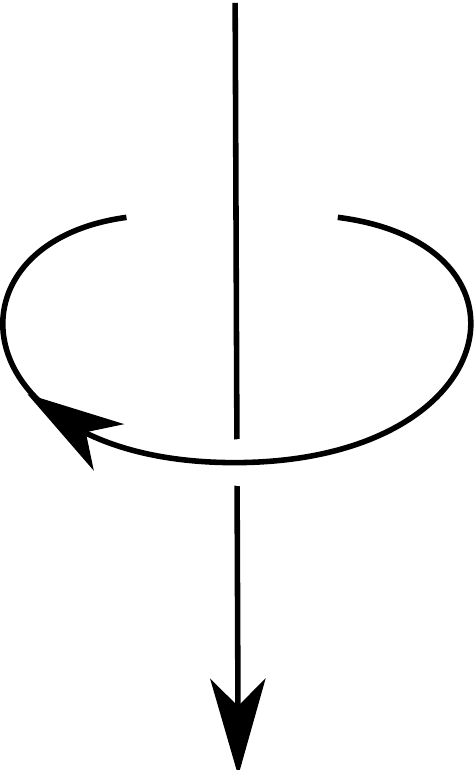}{7ex}
      \put(-6,-8){\ms{V}}
			\put(-11,21){\ms{W}} \right>$$
where the bracket of a diagram $T$ is defined by $F(T)=<T>\Id_W$ with the simple module $W$. Following are some properties of $S'(V,W)$ and the modified dimension $d$ where $d(\mu):=d(V_\mu)$ (for details, see \cite{Ha16}).
\begin{Lem}\label{S' and modified dimension}
Let $V_\mu$ be a typical module and $V'_{\mu'}$ be a simple module with respective highest weight $\mu$ and $\mu'$. Then 
\begin{enumerate}
\item[1.] $d(\mu)=\dfrac{\{\mu_1+1\}}{\ell \{\ell \mu_{1}\}\{\mu_2\}\{\mu_2+\mu_1+1\}}=\dfrac{\{\alpha_{1}\}}{\ell \{\ell \alpha_{1}\}\{\alpha_{2}\}\{\alpha_{1}+\alpha_{2}\}}$ where $(\alpha_{1}, \alpha_{2})=(\mu_{1}-\ell+1, \mu_{2}+\frac{\ell}{2})$.
\item[2.] $S'(V_\mu,V'_{\mu'}):=S'(\mu,\mu')=\xi^{-4\alpha_{2}\alpha'_{2}-2(\alpha_{2}\alpha'_{1}+\alpha_{1}\alpha'_{2})} \dfrac{\{\ell\alpha'_1\}\{\alpha'_2\}\{\alpha'_2+\alpha'_1\}}{\{\alpha'_1\}}$ where $(\alpha'_{1}, \alpha'_{2})=(\mu'_{1}-\ell+1, \mu'_{2}+\frac{\ell}{2})$.
\item[3.] $d(\mu')S'(\mu, \mu')=d(\mu)S'(\mu', \mu).$
\end{enumerate}
\end{Lem}
\begin{proof}[Proof of Theorem \ref{main theorem}]
By Proposition \ref{pre modular}, the category $\mathscr{C}^H$ is pre-modular $G$-category relative to $(\mathit{G}, Z)$. Now we show that this category is a relative modular $G$-category. It is necessary to verify the relative modularity condition. We consider the morphism $f$ which represents by the diagram as in Figure \ref{rel 1}. 
\begin{figure}
   \centering
   $$
   \begin{array}{ccc}
     \epsh{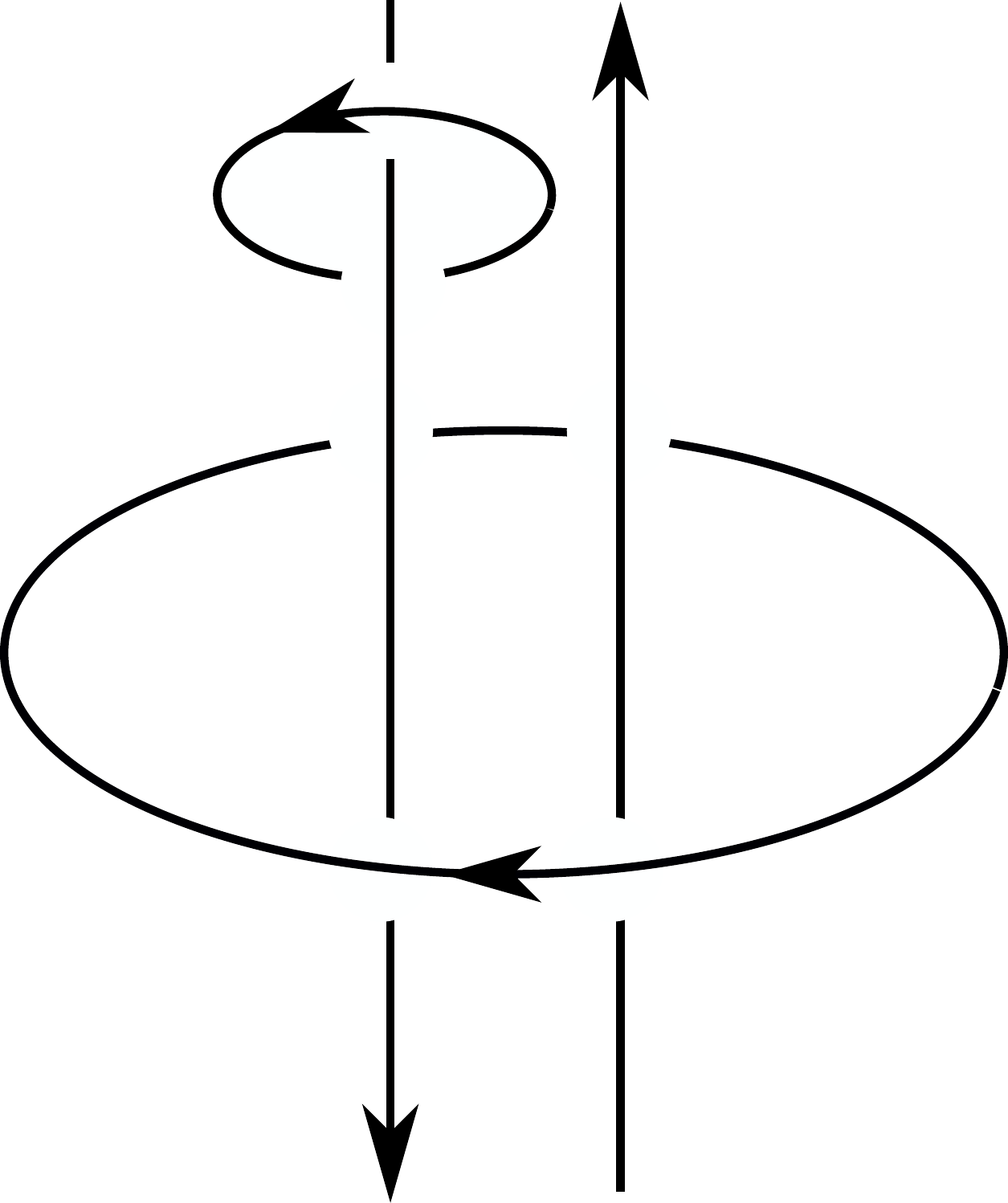}{12ex}
 		\put(3,2){\ms{\Omega_{\overline{\mu}}}}
 		\put(-51, 24){\ms{V_k}}
 		\put(-42,-22){\ms{V_i}}
 		\put(-18,-22){\ms{V_j}}
   \end{array}
   $$
   \caption{Representation of the morphism $f$}
   \label{rel 1}
\end{figure}
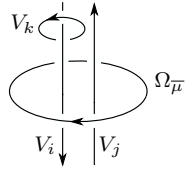
By the handle-slide the circle colored by $V_k$ along the circle of $f_{ij}^{\overline{\mu}}$ and an isotopy we have two equalities given by the diagrams as in Figure \ref{rel 2}.
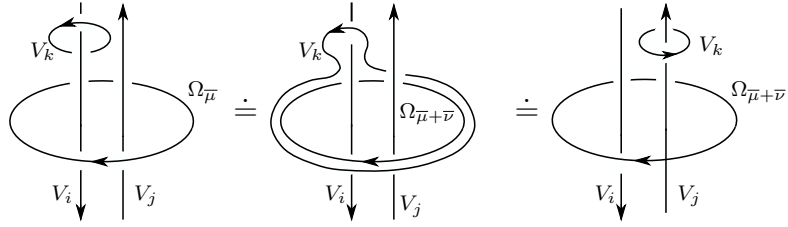
\begin{figure}
   \centering
   $$
   \begin{array}{ccc}
     \epsh{fig32}{16ex}
 		\put(-2,9){\ms{\Omega_{\overline{\mu}}}}
 		\put(-61, 24){\ms{V_k}}
 		\put(-53,-30){\ms{V_i}}
 		\put(-22,-30){\ms{V_j}}
 	\quad  \doteq  \epsh{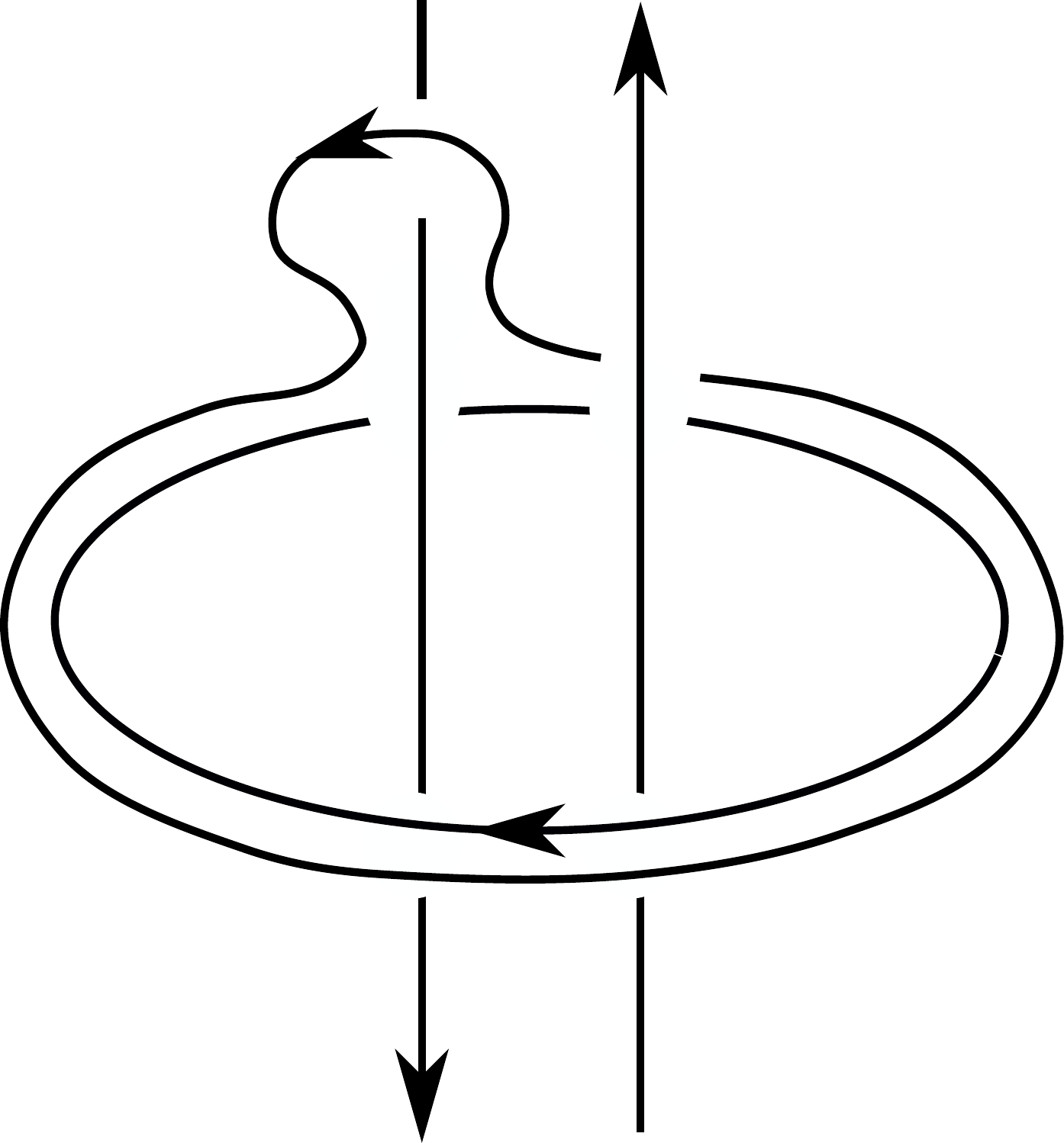}{16ex}
 		\put(-28,1){\ms{\Omega_{\overline{\mu}+\overline{\nu}}}}
 		\put(-66, 24){\ms{V_k}}
 		\put(-56,-30){\ms{V_i}}
 		\put(-28,-34){\ms{V_j}}
 	\quad  \doteq  \epsh{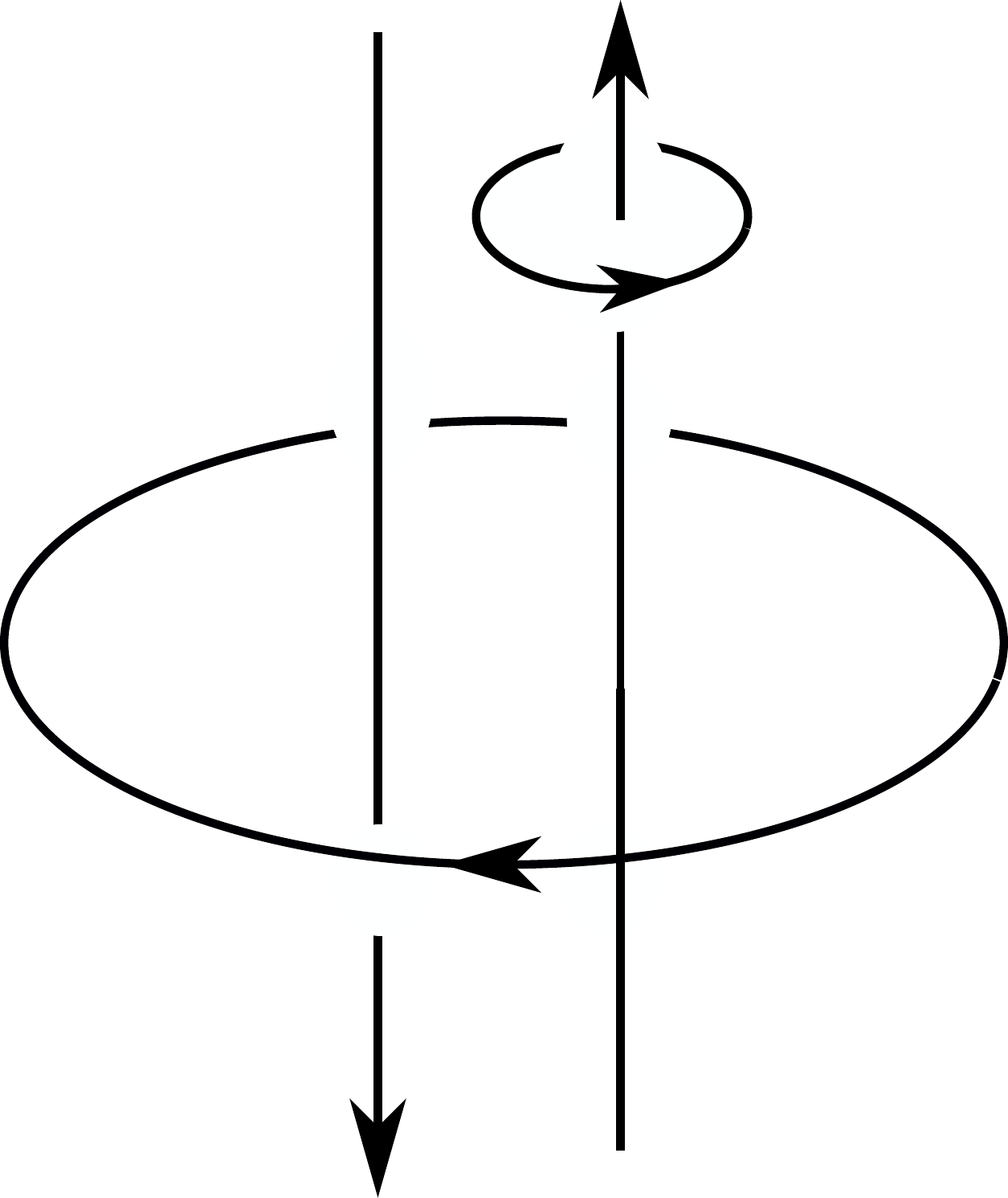}{16ex}
 	\put(-2,9){\ms{\Omega_{\overline{\mu}+\overline{\nu}}}}
 		\put(-14, 26){\ms{V_k}}
 		\put(-53,-30){\ms{V_i}}
 		\put(-22,-30){\ms{V_j}}
   \end{array}
   $$
   \caption{Sliding of the circle colored by $V_k$ along the closed component of $f_{ij}^{\overline{\mu}}$}
   \label{rel 2}
\end{figure}
It follows that $$S'(V_k, V_i)f_{ij}^{\overline{\mu}}=S'(V_k, V_j)f_{ij}^{\overline{\mu}+\overline{\nu}} \ \text{for all}\ V_k \in \mathscr{C}_{\overline{\nu}}^H.$$ 
It implies 
$$f_{ij}^{\overline{\mu}+\overline{\nu}}=\dfrac{S'(V_{k_1}, V_i)}{S'(V_{k_1}, V_j)}f_{ij}^{\overline{\mu}}=\dfrac{S'(V_{k_2}, V_i)}{S'(V_{k_2}, V_j)}f_{ij}^{\overline{\mu}} \ \text{for} \ V_{k_1}, V_{k_2} \in \mathscr{C}_{\overline{\nu}}^H.$$
We denote the highest weights of $V_i, V_j, V_{k_1}$ and $V_{k_2}$ by $(\nu_1 +i_1, \nu_2+i_2), (\nu_1 +j_1, \nu_2+j_2), (\nu_1 + s_1, \nu_2+s_2)$ and $(\nu_1 + t_1, \nu_2+t_2)$ for $0\leq i_1, i_2, j_1, j_2, s_1, s_2, t_1, t_2\leq\ell -1$. By Lemma \ref{S' and modified dimension} we have
\begin{align*}
&S'(V_{k_1}, V_i)=\xi^{-4(\nu_2+s_2)(\nu_2+i_2)-2\left( (\nu_2+s_2)(\nu_1 +i_1)+(\nu_1 + s_1)(\nu_2+i_2)\right)}\frac{1}{\ell d(V_i)},\\
&S'(V_{k_1}, V_j)=\xi^{-4(\nu_2+s_2)(\nu_2+j_2)-2\left( (\nu_2+s_2)(\nu_1 +j_1)+(\nu_1 + s_1)(\nu_2+j_2)\right)}\frac{1}{\ell d(V_j)},\\
&S'(V_{k_2}, V_i)=\xi^{-4(\nu_2+t_2)(\nu_2+i_2)-2\left( (\nu_2+t_2)(\nu_1 +i_1)+(\nu_1 + t_1)(\nu_2+i_2)\right)}\frac{1}{\ell d(V_i)},\\
&S'(V_{k_2}, V_j)=\xi^{-4(\nu_2+t_2)(\nu_2+j_2)-2\left( (\nu_2+t_2)(\nu_1 +j_1)+(\nu_1 + t_1)(\nu_2+j_2)\right)}\frac{1}{\ell d(V_j)}.
\end{align*}
Hence 
\begin{align*}
&\dfrac{S'(V_{k_1}, V_i)}{S'(V_{k_1}, V_j)}=\xi^{-4(\nu_2+s_2)(i_2-j_2)-2\left( (\nu_2+s_2)(i_1-j_1)+(\nu_1 + s_1)(i_2-j_2)\right)}\dfrac{d(V_j)}{d(V_i)}, \\
&\dfrac{S'(V_{k_2}, V_i)}{S'(V_{k_2}, V_j)}=\xi^{-4(\nu_2+t_2)(i_2-j_2)-2\left( (\nu_2+t_2)(i_1-j_1)+(\nu_1 + t_1)(i_2-j_2)\right)}\dfrac{d(V_j)}{d(V_i)}.
\end{align*}
We see that $$\dfrac{S'(V_{k_1}, V_i)}{S'(V_{k_1}, V_j)}.\dfrac{S'(V_{k_2}, V_j)}{S'(V_{k_2}, V_i)}=\xi^{-4(s_2-t_2)(i_2-j_2)-2\left( (s_2-t_2)(i_1-j_1)+(s_1-t_1)(i_2-j_2)\right)}$$
and the term $-4(s_2-t_2)(i_2-j_2)-2\left( (s_2-t_2)(i_1-j_1)+(s_1-t_1)(i_2-j_2)\right)$ is determined by a symmetric bilinear non-degenerate $B$ from $\left(\mathbb{Z}/\ell \mathbb{Z}\right)^2 \times \left(\mathbb{Z}/\ell \mathbb{Z}\right)^2$ to $\mathbb{Z}/\ell \mathbb{Z}$ which has the matrix $B=(b_{ij})_{2\times 2}$ where $b_{11}=0,\ b_{12}=b_{21}=-2$ and $b_{22}=-4$. It follows that for all $i\neq j \in \left(\mathbb{Z}/\ell \mathbb{Z}\right)^2$ it exists $k_1 \neq k_2 \in \left(\mathbb{Z}/\ell \mathbb{Z}\right)^2$ such that $B(i-j, k_1-k_2)\neq 0$. Thus for all $i \neq j \in \overline{\nu}$ it exists $k_1 \neq k_2 \in \overline{\nu}$ such that $\dfrac{S'(V_{k_1}, V_i)}{S'(V_{k_1}, V_j)}\neq\dfrac{S'(V_{k_2}, V_i)}{S'(V_{k_2}, V_j)}$, it implies that $f_{ij}^{\overline{\mu}}=0$ if $i\neq j$.

If $i=j$ we have $f_{ii}^{\overline{\mu}}=f_{ii}^{\overline{\mu}+\overline{\nu}}$ for $\overline{\mu},\overline{\nu}\in \mathit{G}\setminus \mathcal{X}$. We see by Lemma \ref{vector UH-invariant} that $f_{ii}^{\overline{\mu}} \in \End_{\mathcal{U}^H}(V_i \otimes V_{i}^{*})$ factors through invariant vectors of $W=V_i \otimes V_{i}^*$. As $\Hom_{\mathcal{U}^H}(V_i \otimes V_{i}^{*}, \mathbb{C})\simeq \Hom_{\mathcal{U}^H}( \mathbb{C}, V_i \otimes V_{i}^{*})\simeq \End_{\mathcal{U}^H}(V_i)\simeq\mathbb{C}\Id_{V_i}$ then these imply that two morphisms $f_{ii}^{\overline{\mu}}$ and $\overrightarrow{\coev}_{V_i}\circ \tev_{V_i}$ are proportional, i.e., there is a $\lambda \in \mathbb{C}^*$ such that $$f_{ii}^{\overline{\mu}}=\lambda\left(\overrightarrow{\coev}_{V_i}\circ \tev_{V_i}\right).$$ 
Next we compute $\lambda$ from this equality. We consider the value $F'$ of the braid closure of the graphs associated with this equality. The value associated with $f_{ii}^{\overline{\mu}}$ is 
\begin{align*}
F'\left(
	\begin{array}{ccc}	
      \epsh{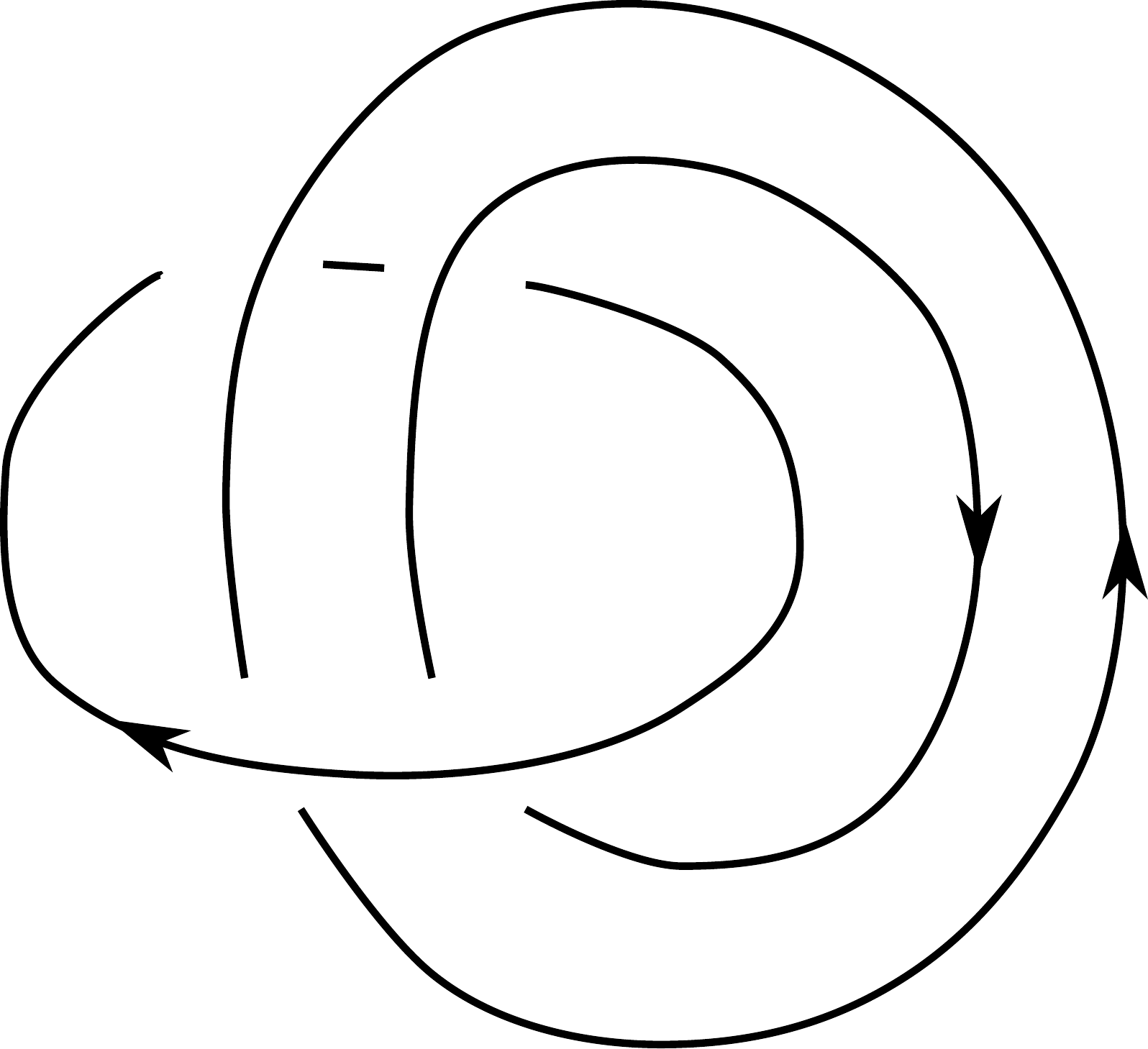}{10ex}
 		\put(-62,-12){\ms{\Omega_{\overline{\mu}}}}
 		\put(-35,2){\ms{V_i}}
 		\put(-20,-30){\ms{V_i}}
 	\end{array}
   \right)
&=\sum_{k}F'\left(d(V_k)
	\begin{array}{ccc}	
      \epsh{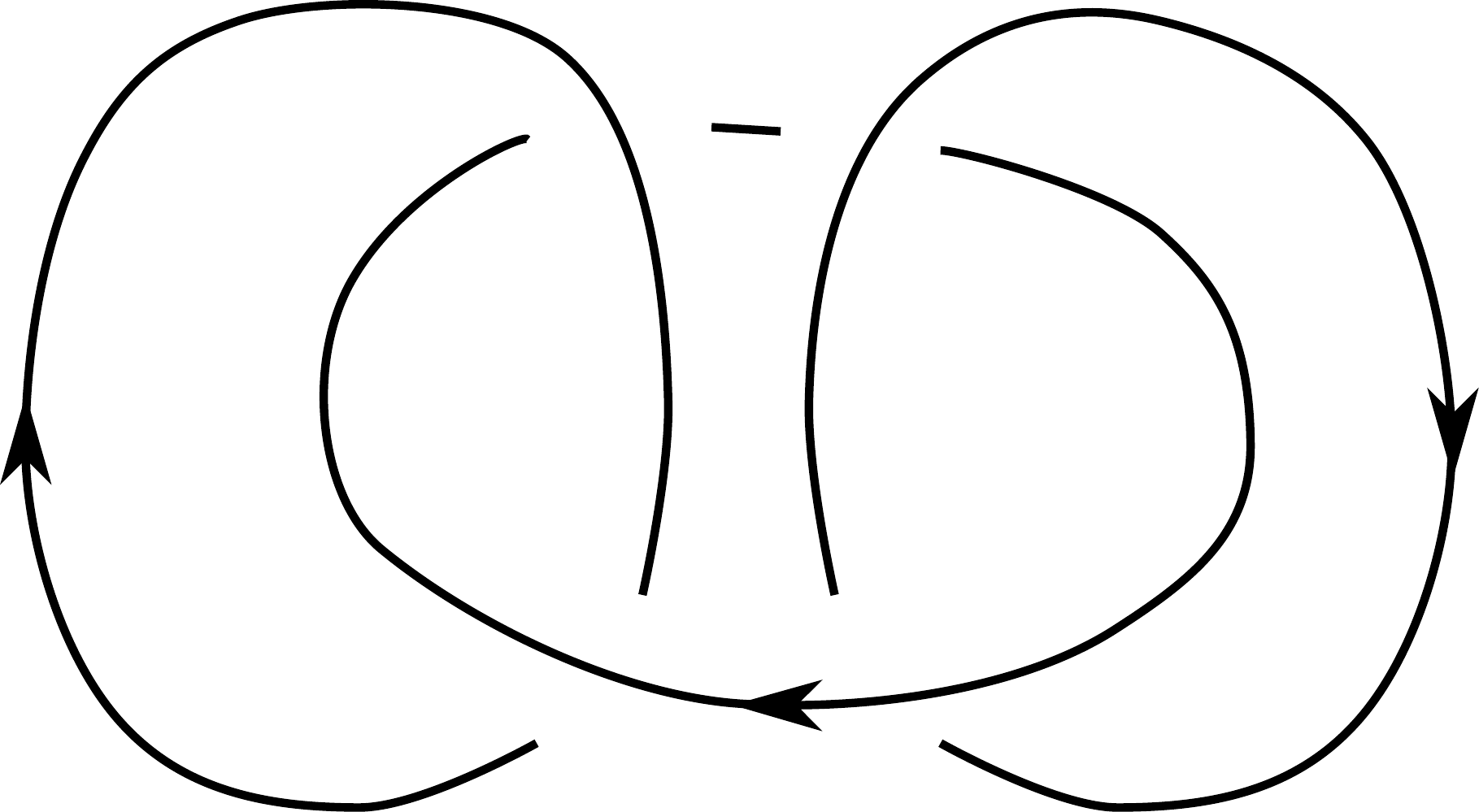}{10ex}
 		\put(-48,-24){\ms{V_k}}
 		\put(-85,-29){\ms{V_i}}
 		\put(-16,-29){\ms{V_i}}
 	\end{array}
   \right)\\
&=\sum_{k}F'\left(d(V_k)
	\begin{array}{ccc}	
      \epsh{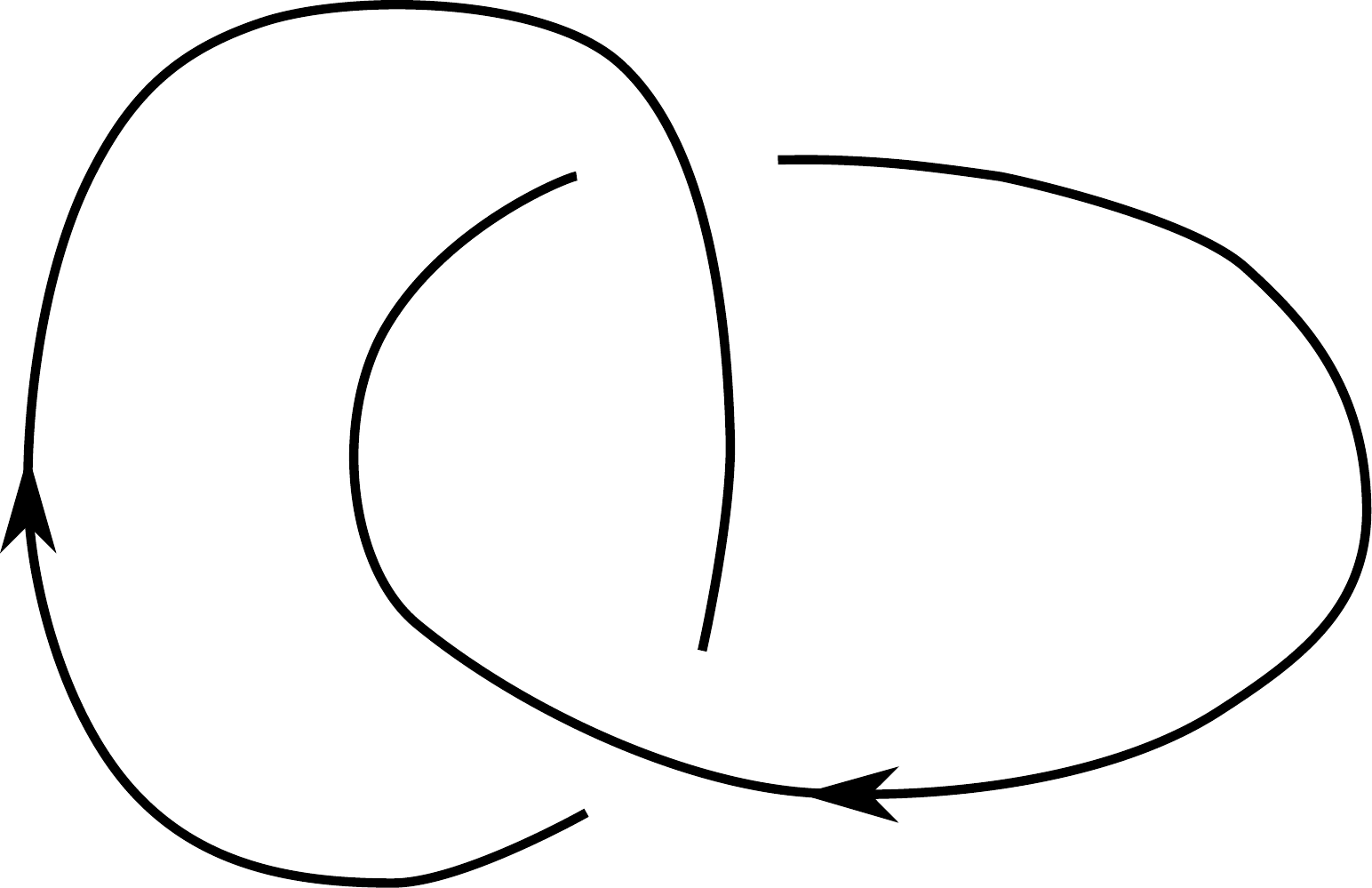}{10ex} \ \# \ \epsh{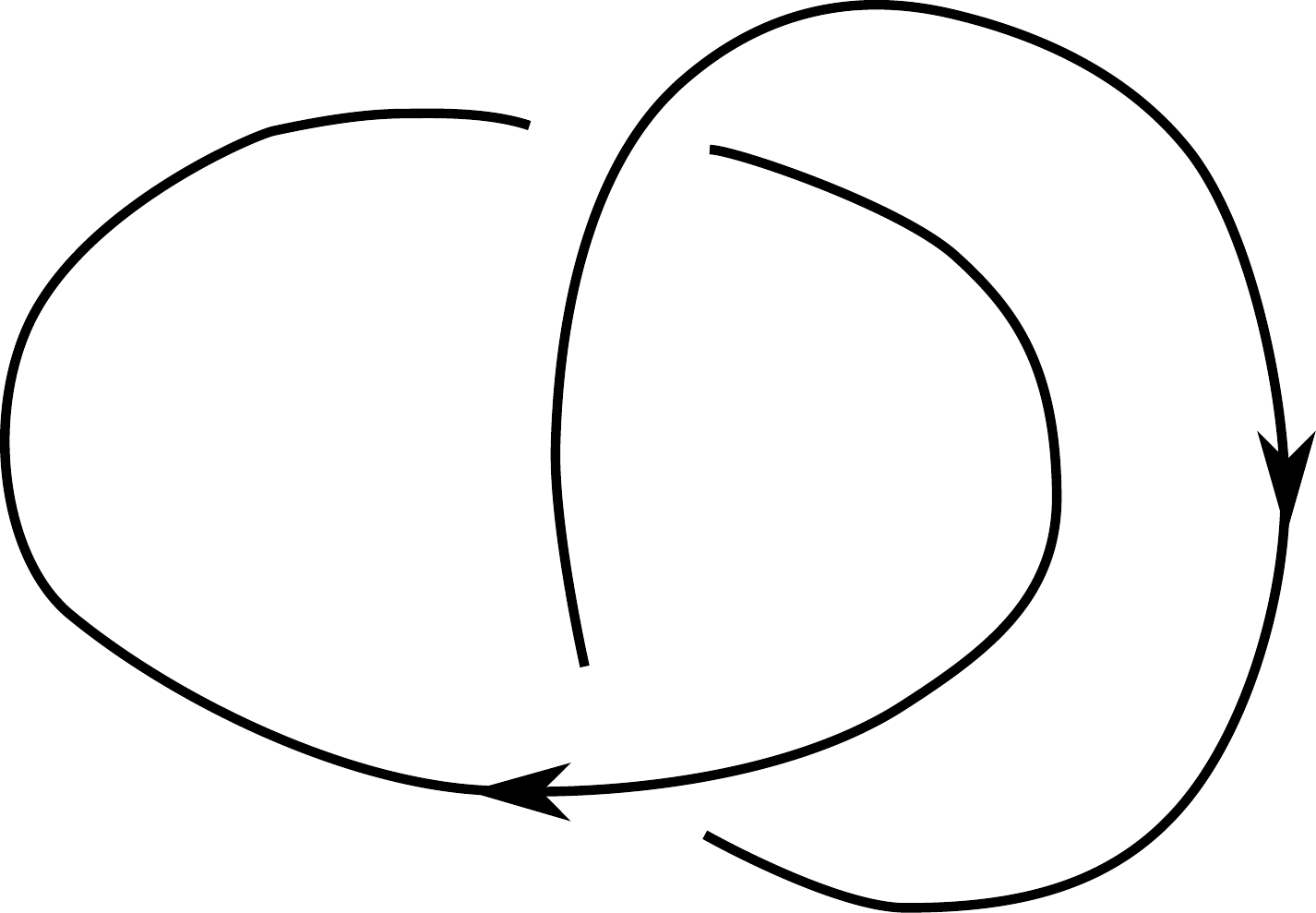}{10ex} 
 		\put(-105,-24){\ms{V_k}}
 		\put(-172,-26){\ms{V_i}}
 		\put(-48,-26){\ms{V_k}}
 		\put(-10,-25){\ms{V_i}}
 	\end{array}
   \right)\\
&=\sum_{k}F'\left(
	\begin{array}{ccc}	
      \epsh{fig37_1}{10ex} 
 		\put(-78,-26){\ms{V_k}}
 		\put(-16,-25){\ms{V_i}} 
 	\end{array}
	\right)
	F'\left(
	\begin{array}{ccc}
		\epsh{fig37_2}{10ex}
 		\put(-68,-21){\ms{V_k}}
 		\put(-8,-25){\ms{V_i}}
 		\end{array}
	\right)\\
&=\sum_{k}d(V_i)S'\left(V_k, V_i\right) d(V_i)S'\left(V_{k}^{*}, V_i\right)\\
&=\sum_{k} d^{2}(V_i)S'\left(V_k, V_i\right) S'\left(V_{k}^{*}, V_i\right)
\end{align*}  
where $\Omega_{\overline{\mu}}=\sum_{k \in \overline{\mu}}d(V_k)V_k$ and the second equality by $F'(L_1 \#_{V} L_2)=d^{-1}(V)F'(L_1)F'(L_2)$. Furthermore, $$S'(V_{k_1}^{*}, V_i)=\xi^{4(\nu_2+s_2)(\nu_2+i_2)+2\left( (\nu_2+s_2)(\nu_1 +i_1)+(\nu_1 + s_1)(\nu_2+i_2)\right)}\frac{1}{\ell d(V_i)},$$
it implies that $$F'\left(
	\begin{array}{ccc}	
      \epsh{fig35}{10ex}
 		\put(-62,-12){\ms{\Omega_{\overline{\mu}}}}
 		\put(-35,2){\ms{V_i}}
 		\put(-20,-30){\ms{V_i}}
 	\end{array}
   \right)=\sum_{s_1, s_2 =0}^{\ell-1}\dfrac{1}{\ell^2}=1.$$
For the graph of $\overrightarrow{\coev}_{V_i}\circ \tev_{V_i}$, the value $F'$ of its closure is
$$F'\left(
\begin{array}{ccc}	
      \epsh{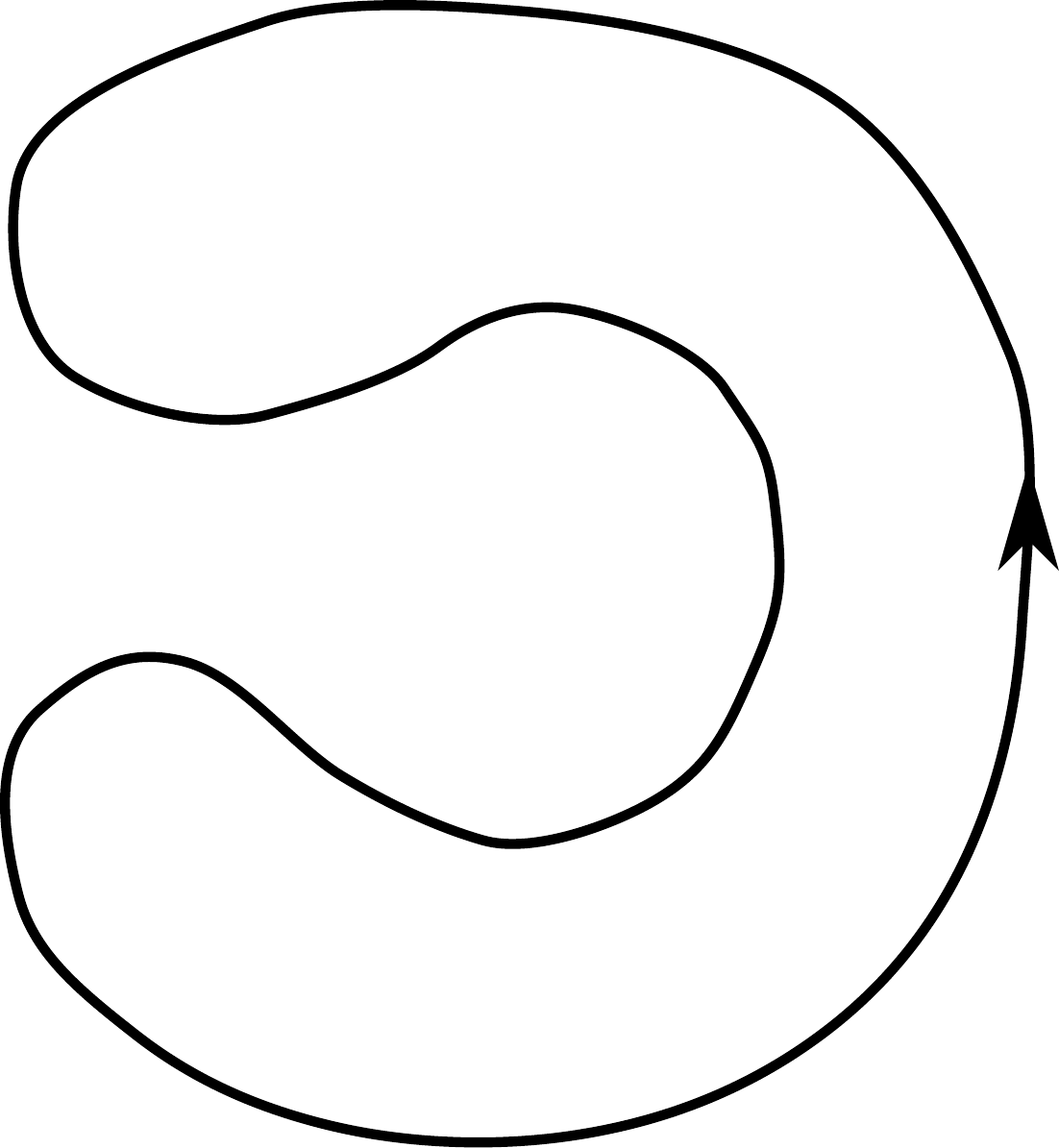}{10ex} 
 		\put(-10,-25){\ms{V_i}} 
 	\end{array}
	\right)
= F'\left(
\begin{array}{ccc}	
      \epsh{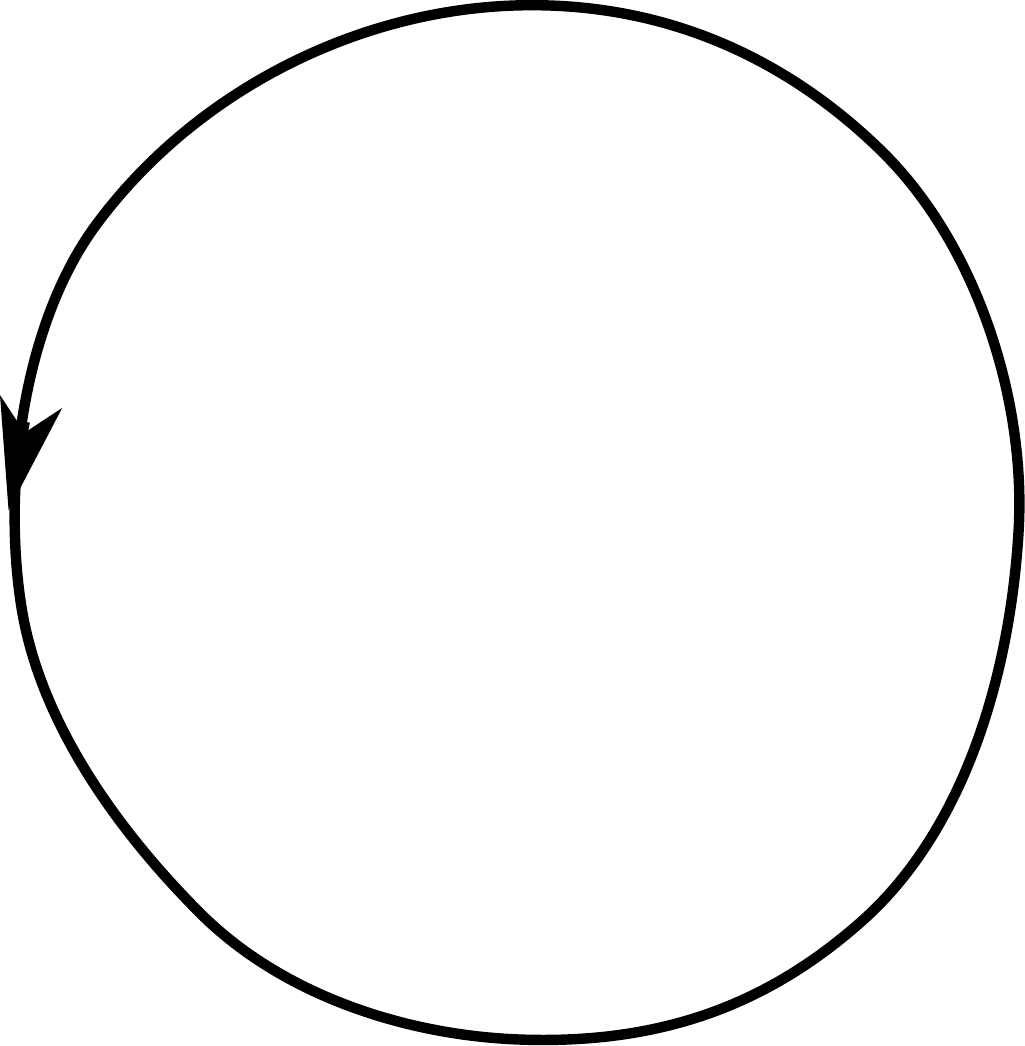}{10ex} 
 		\put(-10,-25){\ms{V_i}} 
 	\end{array}
	\right)
= d(V_i).$$
Hence $\lambda=d^{-1}(V_i)$	and it follows that $d(V_i)f_{ii}^{\overline{\mu}}=\overrightarrow{\coev}_{V_i}\circ \tev_{V_i}$.
\end{proof}
We see that the relative modularity parameter $\zeta=\Delta_-\Delta_+=1$. By \cite[Theorem 1.1]{Mar17}, one gets the corollary.
\begin{HQ}
There exists a family of ETQFTs from category $\mathscr{C}^H$.
\end{HQ}
\begin{rmq}
The anomaly free specificity of this theory implies that the TQFT does not depend on the framing of cobordisms nor on the Lagrangian on surfaces (see \cite{Mar17}). It produces linear (and not only projective as in usual TQFT from quantum groups) representations of the mapping class group.
\end{rmq}
\bibliographystyle{plain} 
\bibliography{TKhao}

\end{document}